\documentclass{article}
\usepackage{amssymb}
\usepackage{amsfonts}
\usepackage{amsmath}

\newtheorem{theorem}{Theorem}
\newtheorem{acknowledgement}[theorem]{Acknowledgement}

\newtheorem{corollary}[theorem]{Corollary}

\newtheorem{lemma}[theorem]{Lemma}

\newtheorem{remark}[theorem]{Remark}

\newenvironment{proof}[1][Proof]{\noindent\textbf{#1.} }{\ \rule{0.5em}{0.5em}}

\input{tcilatex}
\begin{document}

\title{On The Diophantine Equation \\
$x^{2}+7^{\alpha }\cdot 11^{\beta }=\allowbreak y^{n}$}
\date{\today }
\author{G\"{O}KHAN SOYDAN}
\maketitle

\begin{abstract}
In this paper, we give all the solutions of the Diophantine equation $%
x^{2}+7^{\alpha }\cdot 11^{\beta }=\allowbreak y^{n},$ in nonnegative
integers $\alpha ,~\beta ,~x,~y,~n\geq 3$ with $x$ and $y$ coprime, except
for the case when $\alpha .x$ is odd and $\beta ~$is even.
\end{abstract}

\medskip


\medskip


\section{Introduction}

The Diophantine equation 
\begin{equation}
x^{2}+C=y^{n},\qquad n\geq 3  \label{eq:(1.1)}
\end{equation}%
in positive integers $x$,$~y,~n$ for given $C$ has a rich history. In 1850,
Lebesgue \cite{Lebesque} proved that the above equation has no solutions
when $C=1$. The title equation is a special case of the Diophantine equation 
$ay^{2}+by+c=dx^{n},$ where $a\neq 0,b,~c$ and $d\neq 0$ are integers with $%
b^{2}-4ac\neq 0,~$which has at most finitely many integer solutions $%
x,~y,~n\geq 3$ (see \cite{Landau}). \ In 1993, J.H.E. Cohn \cite{Cohn1}
solved the Diophantine equation \eqref{eq:(1.1)} for several values of the
parameter $C$ in the range $1\leq C\leq 100.$ The solution for the cases $%
C=74,86~$ was completed by Mignotte and de Weger \cite{Mignotte} which had
not been covered by Cohn (indeed, Cohn solved these two equations of type %
\eqref{eq:(1.1)} except for $p=5$, in which case difficulties occur as the
class numbers of the corresponding imaginary quadratic fields are divisible
by $5$). In \cite{Bugeaud2},~Bugeaud, Mignotte and Siksek improved modular
methods to solve completely \eqref{eq:(1.1)}$~$when $n\geq 3$, for $C$ in
the range $[1,100]$. So they covered the remaining cases.

Different types of the Diophantine equation \eqref{eq:(1.1)} were studied
also by various mathematicians. For effectively computable upper bounds for
the exponent $n$, we refer to \cite{Berczes2} and \cite{Gyry}. However,
these estimates are based on Baker's theory of lower bounds for linear forms
in logarithms of algebraic numbers, so they are quite impractical. In \cite%
{Tengely1}, Tengely gave a method to solve the equation $x^{2}+a^{2}=y^{n}$
and applied it to $3\leq a\leq 501$, so it includes $x^{2}+7^{2}=y^{n}$ and $%
x^{2}+11^{2}=y^{n}$. In \cite{FSAbu9}, the equation $x^{2}+C=2y^{n}$ with $C 
$ a fixed positive integer and under the similar restrictions $n\geq 3$ and $%
\gcd (x,y)=1$ was studied. Recently, Luca, Tengely and Togb\'{e} studied the
Diophantine equation$~x^{2}+C=4y^{n}$ in nonnegative integers$~x,y$, $n\geq
3 $ with $x$ and $y$ coprime for various shapes of the positive integer $C~$%
in \cite{Luca7}.

In recent years, a different form of the above equation has been considered,
namely where $C$ is a power of a fixed prime. In \cite{Abu1997}, the
equation $x^{2}+2^{k}=y^{n}~$was studied under some conditions by Arif and
Muriefah. A conjecture of Cohn (see \cite{Cohn2}) was verified saying that $%
x^{2}+2^{k}=y^{n}~$has~no~solutions with $x$ odd and even $k>2$ by Le \cite%
{Le1}. In $\cite{Arif1}$, Abu Muriefah and Arif, gave all the solutions\ of $%
x^{2}+3^{k}=y^{n}$with $k$ odd and, Luca \cite{Luca3}, gave all the
solutions with $k$ even. Again the same equation was independently solved in
2008 by Liqun in \cite{Liqun1} for both odd and even $m$. All solutions of \ 
$x^{2}+5^{k}=y^{n}~$are given with k odd in \cite{Arif2} and with k even in 
\cite{FSAbu7}. Liqun solves the same equation again in 2009, in \cite{Liqun2}%
. Recently, B\'{e}rczes and Pink \cite{Berczes}, gave all the solutions of
the Diophantine equation \eqref{eq:(1.1)} when $C=p^{k}$ and $k$ is even,
where $p$ is any prime in the interval $[2,100]$.

The last variant of the Diophantine equation \eqref{eq:(1.1)} in which $C$
is a product of at least two prime powers were studied in some recent
papers. In 2002, Luca gave complete solution of $x^{2}+2^{a}.3^{b}=y^{n}~$in 
$\cite{Luca2}$. Since then, in 2006, all the solutions of the Diophantine
equation $x^{2}+2^{a}.5^{b}=y^{n}$ were found by Luca and Togb\'{e} in $\cite%
{Luca2}$. In 2008, the equations $x^{2}+5^{a}.13^{b}=y^{n}$ and $%
x^{2}+2^{a}5^{b}.13^{c}=y^{n}$ were solved in $\cite{FSAbu3}$ and $\cite%
{Luca5}$. Recently, in \cite{CDLPS} and \cite{CDILS}, complete solutions of
the equations$~x^{2}+2^{a}.11^{b}=y^{n}~$and$%
~x^{2}+2^{a}.3^{b}.11^{c}=y^{n}~ $were found. In \cite{CDST}, the complete
solution $(n,a,b,x,y)~$of the equation $x^{2}+5^{a}.11^{b}=y^{n}$~when $\gcd
(x,y)=1,~$except for the case when $xab$ is odd is given. In $\cite{Pink}$,$%
~ $Pink gave all the \textit{non-exceptional solutions} (in the terminology
of that paper) with $C=2^{a}.3^{b}.5^{c}.7^{d}.$ Note that finding all the 
\textit{exceptional solutions} of this equation seems to be a very difficult
task. A more exhaustive survey on this type of problems is $\cite{FSAbu8}.$

Here, we study the Diophantine equation 
\begin{equation}
x^{2}+7^{\alpha }\cdot 11^{\beta }=y^{n},\qquad gcd(x,y)=1\qquad {\text{%
\textrm{and}}}\qquad n\geq 3.  \label{eq:(1.2)}
\end{equation}%
There are three papers concerned with partial solutions for equation %
\eqref{eq:(1.2)}. The known results include the following theorem:

\begin{theorem}
\label{Theorem:1}$(i)~$If $\alpha $ is even and $\beta =0,~$then the only
integer solutions of the Diophantine equation%
\begin{equation*}
x^{2}+7^{2k}=y^{n}
\end{equation*}%
are%
\begin{eqnarray*}
n &=&3~~\ (x,y,k)=(524\cdot 7^{3\lambda },65\cdot 7^{2\lambda },1+3\lambda ),
\\
n &=&4\ \ \ (x,y,k)=(24\cdot 7^{2\lambda },5\cdot 7^{\lambda },1+2\lambda )\ 
\text{where}\ \lambda \geq 0~\text{is any integer.}
\end{eqnarray*}%
$(ii)$ If $\alpha =1$ and $\beta =0,$ then the only integer solutions $%
(x,y,n)$ to the generalized Ramanujan--Nagell equation%
\begin{equation*}
x^{2}+7=y^{n}
\end{equation*}
\end{theorem}

are%
\begin{equation*}
(1,2,3),(181,32,3),(3,2,4),(5,2,5),(181,8,5),(11,2,7),(181,2,15).
\end{equation*}

$(iii)~$If $\alpha =0,~$then the only integer solutions of the Diophantine
equation%
\begin{equation*}
x^{2}+11^{\beta }=y^{n}
\end{equation*}%
are%
\begin{equation*}
(x,y,\beta ,n)~=(2,5,2,3),(4,3,1,3),(58,15,1,3),(9324,443,3,3)
\end{equation*}

\begin{proof}
See \cite{Luca4}, \cite{Bugeaud2}~and \cite{CDLPS}.
\end{proof}

Our main result is the following.

\begin{theorem}
\label{Theorem:2}The only solutions of the Diophantine equation %
\eqref{eq:(1.2)} are 
\begin{eqnarray*}
n &=&3:\ \ \ (x,y,\alpha ,\beta )\in
\{(57,16,1,2),(797,86,1,2),(4229,284,3,4),(3093,478,7,2), \\
&&\ \ \ \ \ \ \ \ \
(4,3,0,1),(58,15,0,1),(2,5,0,2),(9324,443,0,3),(1,2,1,0),(181,32,1,0), \\
&&~\ \ \ \ \ \ \ \ (524,65,2,0),(13,8,3,0)\}; \\
n &=&4:\text{ \ \ }(x,y,\alpha ,\beta )\in
\{(2,3,1,1),(57,8,1,2),(8343,92,5,2),(3,2,1,0),(24,5,2,0)\}; \\
n &=&6:\text{ \ \ }(x,y,\alpha ,\beta )=(57,4,1,2); \\
n &=&9:\text{ \ \ }(x,y,\alpha ,\beta )=(13,2,3,0); \\
n &=&12:\text{ }(x,y,\alpha ,\beta )=(57,2,1,2);
\end{eqnarray*}%
When $n\geq 5,n\neq 6,9,12$, the equation \eqref{eq:(1.2)} has no solutions $%
(x,y,\alpha ,\beta )$ with at least one of $\alpha ,x$ even or with $\beta ~$%
is odd.
\end{theorem}

\begin{remark}
For $n\geq 5,n\neq 6,9,12$ the above theorem lefts out the solutions $%
(\alpha ,\beta ,x,y)$ when $\alpha .x$ is odd and $\beta ~$is even. These
are exactly the exceptional solutions of the equation \eqref{eq:(1.2)} in
the terminology of$~\cite{Pink}$; see also the remark \ref{Remark:2} at the
end of this paper.
\end{remark}

One can deduce from the Theorem \ref{Theorem:1} and Theorem \ref{Theorem:2}
the following corollary.

\begin{corollary}
The only integer solutions of the Diophantine equation \eqref{eq:(1.2)} are 
\begin{eqnarray*}
n &=&3:\ \ \ (x,y,\alpha ,\beta )\in
\{(57,16,1,2),(797,86,1,2),(4229,284,3,4),(3093,478,7,2), \\
&&\ \ \ \ \ \ \ \ \ (4,3,0,1),(58,15,0,1),(2,5,0,2),\ (9324,443,0,3),\
(1,2,1,0), \\
&&~\ \ \ \ \ \ \ \ (181,32,1,0),(524,65,2,0),(13,8,3,0)\}; \\
n &=&4:\text{ \ \ }(x,y,\alpha ,\beta )\in
\{(2,3,1,1),(57,8,1,2),(8343,92,5,2),(3,2,1,0),(24,5,2,0)\}; \\
n &=&5:\ \ \ (x,y,\alpha ,\beta )=(5,2,1,0),(181,8,1,0); \\
n &=&6:\text{ \ \ }(x,y,\alpha ,\beta )=(57,4,1,2); \\
n &=&7:\ \ \ (x,y,\alpha ,\beta )=(11,2,1,0); \\
n &=&9:\text{ \ \ }(x,y,\alpha ,\beta )=(13,2,3,0); \\
n &=&12:\text{ }(x,y,\alpha ,\beta )=(57,2,1,2); \\
n &=&15:\ (x,y,\alpha ,\beta )=(181,2,1,0).
\end{eqnarray*}
\end{corollary}

\section{The Proof of Theorem$~2$}

We distinguish the cases $n=3,6,9,12,~$ $n=4~$and $n>4$, devoting a
subsection to the treatment of each case.We first treat the cases $n=3$ and $%
n=4$. This is achieved in Section 2.1 and Section 2.2, respectively. For the
case $n=3$, we transform equation \eqref{eq:(1.2)} into several elliptic
equations in Weierstrass form which we need to determine all their $%
\{7,11\}- $integral points. In Section 2.2, we use the same method as in
Section 2.1 to determine the solutions of \eqref{eq:(1.2)}~for $n=4$. In the
last section, we assume that $n>4$ is prime and study the equation %
\eqref{eq:(1.2)} under this assumption. Here we use the method of primitive
divisors for Lucas sequences. All the computations are done with MAGMA \cite%
{Bosma} and with Cremona's program mwrank.

\subsection{The Cases $n=3,6,9~$and $12$}

\begin{lemma}
When $n=3,~$then only solutions to equation $\eqref{eq:(1.2)}$ are%
\begin{eqnarray}
&&(57,16,1,2),(797,86,1,2),(4229,284,3,4),(3093,478,7,2),  \label{eq:(2.1)}
\\
&&(4,3,0,1),(58,15,0,1),(2,5,0,2),(9324,443,0,3),  \notag \\
&&(1,2,1,0),(181,32,1,0),(524,65,2,0),(13,8,3,0);  \notag
\end{eqnarray}%
when $n=6,$ then only solution to equation $\eqref{eq:(1.2)}$ is $%
(57,4,1,2); $~when $n=9,$ then only solution to equation $\eqref{eq:(1.2)}$
is $(13,2,3,0);$when $n=12$, then only solution to equation $\eqref{eq:(1.2)}
$ is $(57,2,1,2).$
\end{lemma}

\begin{proof}
Suppose $n=3$.~Writing$~\alpha =6k+\alpha _{1},~\beta =6l+\beta _{1}~$in 
\textit{\eqref{eq:(1.2)} }with$~\alpha _{1},\beta _{1}\in \{0,1,2,3,4,5\},$
we get that%
\begin{equation*}
\left( \frac{x}{7^{3k}11^{3l}},\frac{y}{7^{2k}11^{2l}}\right)
\end{equation*}%
is an $S-$Integral point $(X,Y)~$on the elliptic curve 
\begin{equation}
X^{2}=Y^{3}-7^{\alpha _{1}}\cdot 11^{\beta _{1}},  \label{eq:(2.2)}
\end{equation}%
where$~$ $S=\{7,11\}~$with the \ numerator of$~Y$~being coprime to $77,$ in
view of the restriction $\gcd (x,y)=1.$ Now we need to determine all the $%
\{7,11\}$-integral$~$points on the above $36~$elliptic curves. At this stage
we note that in~\cite{Gebel} Peth\H{o}, Zimmer, Gebel and Herrmann developed
a practical method for computing all$~S-$Integral points on Weierstrass
elliptic curve and their method has been implemented in MAGMA \cite{Bosma}
as a routine under the name \texttt{SIntegralPoints.}The subroutine \texttt{%
SIntegralPoints }of~MAGMA worked without problems for all $(\alpha
_{1},\beta _{1})~$except for $(\alpha _{1},\beta _{1})=(5,5)$. MAGMA
determined the appropriate Mordell-Weil groups except this case and$~$we
deal with this exceptional case separately. By computations done for
equation \textit{\eqref{eq:(2.2)}} when $n=3$,$~$we obtain the following
solutions for the $\{7,11\}-$integral~points on the curves:%
\begin{eqnarray*}
&&(1,0,0,0),(3,4,0,1),(15,58,0,1),(5,2,0,2),(11,0,0,3),(443,9324,0,3), \\
&&(2,1,1,0),(32,181,1,0),~(478/49,3093/3431,2),(11,22,1,2),(16,57,1,2), \\
&&(1899062/117649,2338713355/40353607,1,2),(22,99,1,2),(86,797,1,2), \\
&&(88,825,1,2),(638,16115,1,2),(657547,533200074,1,2),(242,3751,1,4), \\
&&(65,524,2,0),(7,0,3,0),(8,13,3,0),(14,49,3,0),(28,147,3,0), \\
&&(154,1911,3,0),(77,0,3,3),(242,3025,3,4),(284,4229,3,4), \\
&&(1435907/49,1720637666/343,3,4).
\end{eqnarray*}%
We use the above points on the elliptic curves to find the corresponding
solutions for equation \textit{\eqref{eq:(2.2)}}. Identifying the coprime
positive integers $x$ and $y$ from the above list, one obtains the solutions
listed in \textit{\eqref{eq:(2.2)}} (note that not all of them lead to
coprime values for $x$ and $y$).

We give the details in case $(\alpha _{1},\beta _{1})=(5,5)~$of equation 
\textit{\eqref{eq:(2.2)}}. Observe that if$~Y~$is even, then $X~$is odd and $%
X^{2}+$ $7^{5}11^{5}\equiv 0\pmod 8$, and hence $X^{2}\equiv 3\pmod 8,~$%
which is a contradiction. Therefore $Y$ is always odd. We consider solutions
such that $X$ and $Y$ are coprime.

Write ${\mathbb{K}}=\mathbb{Q}(i\sqrt{77}).~$In this field,$~$the primes $%
2,7,11~$(all primes dividing discriminant $d_{{\mathbb{K}}}=4d~$) ramify so
there are prime ideals $P_{2},P_{7},P_{11}~$such that $2{\mathcal{O}}_{%
\mathbb{K}}=P_{2}^{2}~,$ $7{\mathcal{O}}_{\mathbb{K}}=P_{7}^{2},$ $11{%
\mathcal{O}}_{\mathbb{K}}=P_{11}^{2}~$respectively. Now, we show that the
ideals $(X+7^{2}11^{2}\sqrt{77}i){\mathcal{O}}_{\mathbb{K}}~$and $%
(X-7^{2}11^{2}\sqrt{77}i){\mathcal{O}}_{\mathbb{K}}~$\ are coprime in the
ring of integers ${\mathcal{O}}_{\mathbb{K}}$ $.~$To show this, let us
assume that the ideals $(X+7^{2}11^{2}\sqrt{77}i){\mathcal{O}}_{\mathbb{K}}~$%
and $(X-7^{2}11^{2}\sqrt{77}i){\mathcal{O}}_{\mathbb{K}}$ are not coprime.
So, these ideals have a $\gcd ~$that divides $2.7^{2}.11^{2}\sqrt{77}i.~$%
Hence there is an ideal $P_{2}^{a}P_{7}^{b}P_{11}^{c}~$with $a\leq 2$, and $%
b,c\leq 5~$. If $b>0~$then $7~|~X$.$~$Hence$~7\ |~Y$, hence $7^{3}~|$ $X^{2}$%
, hence $7^{2}~|~X$, hence~$7^{4}|~Y^{3}$, hence $7^{2}$~$|~Y$, hence $%
7^{5}~|$ $X^{2}$, hence $7^{3}|$ $X.~$So, we have a contradiction as $%
7^{6}~|~X^{2}-Y^{3}.$ Thus $b=0.$ Similarly we can prove that $c=0.~$

Now let $(X+7^{2}11^{2}\sqrt{77}i){\mathcal{O}}_{\mathbb{K}}=P_{2}^{a}\wp
^{3}~$for some ideal $\wp $ not divisible by $P_{2}$, and $(X-7^{2}11^{2}%
\sqrt{77}i){\mathcal{O}}_{\mathbb{K}}=P_{2}^{a}\wp ^{\prime 3}$ (for its
conjugate ideal). If we take norms, then we get that $y^{3}=2^{a}[N_{\mathbb{%
K}}(\wp )]^{3}$, where $N_{\mathbb{K}}(\wp )$ is odd. It follows that $a=0~$%
(as it could be at most $2$). So, we showed that the ideals $(X+7^{2}11^{2}%
\sqrt{77}i){\mathcal{O}}_{\mathbb{K}}~$and $(X-7^{2}11^{2}\sqrt{77}i){%
\mathcal{O}}_{\mathbb{K}}$ are coprime. Equation \textit{\eqref{eq:(2.2)} }%
now implies that%
\begin{equation*}
(X+7^{2}11^{2}\sqrt{77}i){\mathcal{O}}_{\mathbb{K}}=\wp ^{3}~\text{and~}%
(X-7^{2}11^{2}\sqrt{77}i){\mathcal{O}}_{\mathbb{K}}~=\wp ^{\prime 3}
\end{equation*}%
for the ideals $\wp ~$and $\wp ^{\prime }$. Let $h({\mathbb{K)}}$ be the
class number of the field ${\mathbb{K}}$, then $\delta ^{h({\mathbb{K)}}}~$%
is principal for any ideal $\delta .~$Note~that, $h({\mathbb{K)=}}8~$and so$%
~(3,h({\mathbb{K))=}}1.~$Thus$~$since $\wp ^{3}$ and $\wp ^{\prime 3}$ are
principal, $\wp ~$and $\wp ^{\prime }$ are also principal.~Moreover, since
the units of $\mathbb{Q}(i\sqrt{77})$ are $1$ and $-1,$ which are both
cubes, we conclude that%
\begin{eqnarray}
(X+7^{2}11^{2}\sqrt{77}i) &=&(u+\sqrt{77}iv)^{3}\text{ }  \label{eq:(2.3)} \\
(X-7^{2}11^{2}\sqrt{77}i) &=&(u-\sqrt{77}iv)^{3}\text{ }  \label{eq:(2.4)}
\end{eqnarray}%
for some integers $u~$and $v.~$After subtracting the conjugate equation we
obtain%
\begin{equation}
~7^{2}\cdot 11^{2}=v(3u^{2}v-77v^{2}).  \label{eq:(2.5)}
\end{equation}%
Since $u~$and $v$ are coprime, we have the following possibilities in
equation \textit{\eqref{eq:(2.5)}}%
\begin{equation*}
~v=\pm 1;~v=\pm 7^{2};~v=\pm 11^{2};~v=\pm 7^{2}11^{2}
\end{equation*}%
All cases lead to the conclusion that no solution is obtained.

For $n=6,$ equation%
\begin{equation*}
x^{2}+7^{\alpha }\cdot 11^{\beta }=y^{6}
\end{equation*}%
becomes equation%
\begin{equation*}
x^{2}+7^{\alpha }\cdot 11^{\beta }=(y^{2})^{3}.
\end{equation*}%
Again, here we look in the list of solutions of equation \textit{%
\eqref{eq:(2.1)} }and observe that the only solution whose\textit{\ }$\ y$
is a perfect square is $(57,16,1,2).$Therefore the only solution to equation 
\textit{\eqref{eq:(1.2)}} is $(57,4,1,2).$ In the same way, one can see that
the value of $y$ above which is a perfect square is $y=4$ for the solution $%
(57,4,1,2)$, therefore the only solution with $n=12$ is $(57,2,1,2)$.

For $n=9,$ equation%
\begin{equation*}
x^{2}+7^{\alpha }\cdot 11^{\beta }=y^{9}
\end{equation*}%
becomes equation%
\begin{equation*}
x^{2}+7^{\alpha }\cdot 11^{\beta }=(y^{3})^{3}.
\end{equation*}%
Again here, we look in the list of solutions of \textit{\eqref{eq:(2.1)} }%
and observe that only solution whose $y$ is a perfect cube is $(13,8,3,0).$%
Therefore the only solution to equation \textit{\eqref{eq:(1.2)}} is $%
(13,2,3,0).$This completes the proof of lemma.
\end{proof}

If $(x,y,\alpha ,\beta ,n)$ is a solution of the Diophantine equation %
\eqref{eq:(1.2)} and $d$ is any proper divisor of $n$, then $(x,y^{d},\alpha
,\beta ,n/d)$ is also a solution of the same equation. Since $n>3$ and we
have already dealt with case $n=3,~$it follows that it suffices to look at
the solutions $n$ for which $p~|~n$ for some odd prime $p.~$In this case, we
may certainly replace $n$ by $p$, and thus assume for the rest of the paper
that $n\in \{4,p\}$.

\subsection{The Case $n=4$}

\begin{lemma}
The only solutions with $n=4$ of the Diophantine equation \eqref{eq:(1.2)}
are given by%
\begin{equation*}
(x,y,\alpha ,\beta )=(2,3,1,1),(57,8,1,2),(8343,92,5,2),(3,2,1,0),(24,5,2,0)
\end{equation*}
\end{lemma}

\begin{proof}
Suppose that $~n=4$. Rewrite equation \eqref{eq:(1.2)}~as 
\begin{equation}
7^{\alpha }\cdot 11^{\beta }=(y^{2}+x)(y^{2}-x).  \label{eq:(4.1)}
\end{equation}%
From the equation \eqref{eq:(4.1)}, we have that 
\begin{eqnarray*}
y^{2}+x &=&7^{a_{1}}.11^{b_{1}} \\
y^{2}-x &=&7^{a_{2}}.11^{b_{2}}
\end{eqnarray*}%
where $a_{1},a_{2},b_{1},b_{2}\geq 0.$ Then we get that$~$%
\begin{equation*}
2y^{2}=7^{a_{1}}.11^{b_{1}}+7^{a_{2}}.11^{b_{2}}
\end{equation*}%
from the sum of two equations. We multiply above equation by $2~$and we can
write the equation 
\begin{equation}
Z^{2}=2.(7^{a_{1}}.11^{b_{1}}+7^{a_{2}}.11^{b_{2}})  \label{eq:(4.2)}
\end{equation}%
as%
\begin{equation}
2U+2V=Z^{2}  \label{eq:(4.3)}
\end{equation}%
where $Z=$ $2y,~~U=7^{a_{1}}.11^{b_{1}}$ and $V=7^{a_{2}}.11^{b_{2}}$.

Let $~p_{1},p_{2},...,p_{s}~(s\geq 1)$ be fixed distinct primes. The set of $%
S-$Units is defined as $S=\left\{ \pm
p_{1}^{x_{1}}p_{2}^{x_{2}}...p_{s}^{x_{s}}|~x_{i}\in 
%TCIMACRO{\U{2124} }%
%BeginExpansion
\mathbb{Z}
%EndExpansion
,~\text{for }i=1...k\right\} .~$Let $a,b\in 
%TCIMACRO{\U{211a} }%
%BeginExpansion
\mathbb{Q}
%EndExpansion
-\{0\}$ be fixed. In $\cite{De Weger1}$,~B.M.M. de Weger dealt with the
solutions of the Diophantine equation $ax+by=z^{2}$, in$~a,b\in S,~z\in 
%TCIMACRO{\U{211a} }%
%BeginExpansion
\mathbb{Q}
%EndExpansion
.~$ He showed that this equation has essentially only finitely many
solutions. Moreover, he indicated how to find all the solutions of this
equation for any given set of parameters $a,b,~p_{1},...,p_{s}$. The tools
are the theory of p-adic linear forms in logarithms, and a computational
p-adic diophantine approximation method. He actually performed all the
necessary computations for solving \eqref{eq:(4.3)} completely for $%
p_{1},...,p_{s}=2,3,5,7$ and $a=b=1$, and reported on this elsewhere (see $%
\cite{De Weger2}$, Chapter 7). Then we can find all the solutions of the
Diophantine equation \eqref{eq:(4.2)}. But this requires a lot of additional
manual effort. To solve the equation $x^{2}+7^{\alpha }\cdot 11^{\beta
}=y^{4}~$instead of \ this method, we prefer using MAGMA (see \cite{Bosma}).

Writing in \eqref{eq:(1.2)}\textit{\ }$\alpha =4k+\alpha _{1},~\beta
=4l+\beta _{1}~$with$~\alpha _{1},\beta _{1}\in \{0,1,2,3\}$ we get that%
\begin{equation*}
\left( \frac{x}{7^{2k}11^{2l}},\frac{y}{7^{2k}11^{2l}}\right)
\end{equation*}%
is an $S-$Integral point $(X,Y)~$on the hyperelliptic curve%
\begin{equation}
X^{2}=Y^{4}-7^{\alpha _{1}}\cdot 11^{\beta _{1}},  \label{eq:(4.4)}
\end{equation}%
where$~$ $S=\{7,11\}~$with the \ numerator of$~Y$~being prime to $77,$ in
view of the restriction $\gcd (x,y)=1.$ We use the subroutine \texttt{%
SIntegralLjunggrenPoints} of MAGMA to determine the $\{7,11\}$-integral
points on the above hyperelliptic curves and we only find the following
solutions%
\begin{eqnarray*}
(X,Y,\alpha _{1},\beta _{1}) &=&\{(1,0,0,0),(2,3,1,0),(3,2,1,1),(8,57,1,2),
\\
&&(92/7,8343/49,1,2),(5,24,2,0)\}
\end{eqnarray*}%
With the conditions on $x$ and $y$ and the definition of$~X,Y$, one can
obtain the solutions listed in the statement of the lemma.
\end{proof}

\subsection{The Case $n>4$ and Prime}

\begin{lemma}
\label{Lemma:3}The Diophantine equation \eqref{eq:(1.2)} has no solutions
with $n>4$ prime except possibly for$~\alpha ~$and$~x$ are odd and $\beta $
even.
\end{lemma}

\begin{proof}
Since in section 2 we have finished the study of equation $x^{2}+7^{\alpha
}\cdot 11^{\beta }=y^{n}$ with $n=3,$ we can assume that $n$ is a prime$~>4$%
. One can write the Diophantine equation \eqref{eq:(1.2)} as $%
x^{2}+dz^{2}=y^{n},$ where 
\begin{equation}
d\in \{1,~7,~11,~77\},~~z=7^{\alpha _{1}}\cdot 11^{\beta _{1}}
\label{eq:(3.1)}
\end{equation}%
the relation of $\alpha _{1}~$and$~\beta _{1}$ with $\alpha $ and $\beta ,$
respectively, is clear. If~$x$ is odd, then by $z$ also being odd we have
that $y$ is even, so $y^{n}\equiv 0~\pmod 8$. As$~x^{2}=z^{2}\equiv $ $1$ $%
\pmod 8~$we have $1+d\equiv $ $0~\pmod 8$, so $d=7$, implying $\alpha \equiv
1$ $\pmod 2$ and $\beta \equiv 0$ $\pmod 2$. This case is excluded in the
lemma. Hence we have that $x$ is even, and $y$ is odd. We study in the field 
${\mathbb{K}}=\mathbb{Q}(i\sqrt{d}).~$As $\gcd (x,z)=1$ standard argument
tells us now that in ${\mathbb{K}}$ we have 
\begin{equation}
(x+i\sqrt{d}z)(x-i\sqrt{d}z)=y^{n},  \label{eq:(3.2)}
\end{equation}%
where the ideals generated by $x+iz\sqrt{d}$ and $x-iz\sqrt{d}$ are coprime
in ${\mathbb{K}}$. Hence, we obtain the ideal equation%
\begin{equation}
\langle x+i\sqrt{d}z\rangle =\theta ^{n}  \label{eq:(3.3)}
\end{equation}%
Then, since the ideal class number of ${\mathbb{K}}$ is $1~$or $8,$ and $n$
is odd, we conclude that the ideal $\theta $ is principal. The cardinality
of the group of units of ${\mathcal{O}}_{{\mathbb{K}}}$ is $2~~$or$~4$, all
coprime to $n$. Furthermore, $\{1,i\sqrt{d}\}$ is always an integral base
for ${\mathcal{O}}_{{\mathbb{K}}}$ except for when $d=7$, and $d=11$, in
which cases an integral basis for ${\mathcal{O}}_{{\mathbb{K}}}$ is $\{1,(1+i%
\sqrt{d})/{2}\}$. Thus, we may assume that%
\begin{equation}
x+i\sqrt{d}z=\varphi ^{n},~\varphi =\frac{u+i\sqrt{d}v}{2}  \label{eq:(3.4)}
\end{equation}%
the relation holds with some algebraic integer $\varphi \in {\mathcal{O}}_{{%
\mathbb{K}}}.$ The algebraic integers in this number field are of the form $%
\varphi =\frac{u+i\sqrt{d}v}{2},$ where $u,v\in 
%TCIMACRO{\U{2124} }%
%BeginExpansion
\mathbb{Z}
%EndExpansion
,~$with $u,v$ both even, if $d=1,77$ and $u,v$ both odd if $d=7,11$. Note
that%
\begin{equation*}
\varphi -{\overline{\varphi }}=vi\sqrt{d},~\varphi +{\overline{\varphi }}=i%
\sqrt{d}v,~\varphi {\overline{\varphi }}=\frac{u^{2}+dv^{2}}{4}
\end{equation*}%
We thus obtain%
\begin{equation}
\frac{2\cdot 7^{\alpha _{1}}\cdot 11^{\beta _{1}}}{v}=\frac{2z}{v}=\frac{%
\varphi ^{n}-{\overline{\varphi }}^{n}}{\varphi -{\overline{\varphi }}}\in 
\mathbb{Z}.  \label{eq:(3.5)}
\end{equation}%
Let $(L_{m})_{m\geq 0}$ be the sequence with ge$\mathcal{7}$neral term $%
L_{m}=(\varphi ^{m}-{\overline{\varphi }}^{m})/(\varphi -{\overline{\varphi }%
})$ for all $m\geq 0.$ This is called a \textit{Lucas sequence}$.$ Note that%
\begin{equation}
L_{0}=0,L_{1}=1~\text{and~}L_{m}=uL_{m-1}-\frac{u^{2}+dv^{2}}{4}%
L_{m-2},~m\geq 2.  \label{eq:(3.6)}
\end{equation}%
Following the nowadays standard strategy based on the important paper $\cite%
{Bilu},$ we distinguish two cases according as $L_{n}~$has or has not
primitive divisors.

Suppose first that $L_{n}$ has a primitive divisor, say $q$. By definition,
this means that the prime $q$ divides $L_{n}~$and $q$ does not divide $(\mu -%
{\overline{\mu })}^{2}L_{1}...L_{n-1},$ hence%
\begin{equation}
q\nmid (\varphi -{\overline{\varphi })}^{2}L_{1}...L_{4}=(dv^{2}).u.\frac{%
3u^{2}-dv^{2}}{4}.\frac{u^{2}-dv^{2}}{2}.  \label{eq:(3.7)}
\end{equation}%
\qquad \qquad

If $q=2$, then \textit{\eqref{eq:(3.7)}}~implies that $uv$ is odd, hence $%
d=11~$or $77.$ If~$d=11$, then third factor in the right hand-most side of 
\textit{\eqref{eq:(3.7)}}~is even, a contradiction. If $d=77$, then, from 
\textit{\eqref{eq:(3.6)}}~we see that $L_{m}\equiv L_{m-1}\pmod 2,$ hence $%
L_{m}$ is odd for every $m\geq 1$, implying that $2$ cannot be a primitive
divisor of $L_{n}.$

If $q=7$, then \textit{\eqref{eq:(3.7)}} implies that $d=1,11$ and $7$ does
not divide $uv(3u^{2}-dv^{2})(u^{2}-dv^{2}).$ It follows easily then that $%
v^{2}\equiv -du^{2}\pmod 7,$ so that, by \textit{\eqref{eq:(3.6)}}, $%
L_{m}\equiv uL_{m-1}\pmod 8$ for every $m\geq 2$. Therefore, $7\nmid L_{n}$,
so that $7$ can not be a prime divisor of $L_{n}$.

If $q=11,$ then by \textit{\eqref{eq:(3.7)}}, $d=1$ or $7.~$If $d=1~$then we
write $u=2v_{1},v=2v_{1}$ with $u_{1},v_{1}\in 
%TCIMACRO{\U{2124} }%
%BeginExpansion
\mathbb{Z}
%EndExpansion
$, so that $\varphi =u_{1}+i\sqrt{d}v_{1}$ and \textit{\eqref{eq:(3.7)}}
becomes $q\nmid u_{1}v_{1}(3u_{1}^{2}-dv_{1}^{2})(u_{1}^{2}-dv_{1}^{2}).$
Moreover, $L_{m}=2u_{1}L_{m-1}-(u_{1}^{2}+dv_{1}^{2})L_{m-2}$ for $m\geq 2$.
Note that $\varphi {\overline{\varphi }=u}_{1}^{2}+dv_{1}^{2}\neq 0\pmod 8;$
therefore, by corollary $2.2$ of $\cite{Bilu}$, there exists a positive
integer $m_{11}$ such that $11\mid L_{m_{11}}$ and $m_{11}\mid m$ for every $%
m$ such that$~11\mid L_{m}.~$It follows then that $11\mid \gcd
(L_{n},L_{m_{11}})=L_{\gcd (n,m_{11})}.~$Because of the minimality property
of $m_{11}$, we conclude that $\gcd (n,m_{11}),$ hence, since $n$ is a
prime, $m_{11}=n.$ On the other hand, the Legendre symbol $\left( \frac{%
(\varphi -{\overline{\varphi })}^{2}}{11}\right) =-1,$ hence by Theorem XII
of $\cite{Carmichael}~$(or by theorem 2.2.4 (iv) of $\cite{Luca6}$), $11\mid
L_{12}.~$Therefore $m_{11}\mid 12,~$i.e.~$n\mid 12$, a contradiction, since $%
n$ is a prime$\geq 5.~$If $d=7,$ then \textit{\eqref{eq:(3.7)}} implies $%
11\nmid u_{1}v_{1}(3u_{1}^{2}-dv_{1}^{2})(u_{1}^{2}-dv_{1}^{2}).~$Moreover, $%
L_{m}=2u_{1}L_{m-1}-(u_{1}^{2}+dv_{1}^{2})L_{m-2}$ for $m\geq 2$. Note that $%
\varphi {\overline{\varphi }=u}_{1}^{2}+dv_{1}^{2}\neq 0\pmod 8$; therefore,
by corollary $2.2$ of $\cite{Bilu}$, there exists a positive integer $m_{11}$
such that $11\mid L_{m_{11}}$ and $m_{11}\mid m$ for every $m$ such that$%
~11\mid L_{m}.$It follows then that $11\mid \gcd (L_{n},L_{m_{11}})=L_{\gcd
(n,m_{11})}.$ Because of the minimality property of $m_{11}$, we conclude
that $\gcd (n,m_{11}),$ hence, since $n$ is a prime, $m_{11}=n.$ On the
other hand, the Legendre symbol $\left( \frac{(\varphi -{\overline{\varphi })%
}^{2}}{11}\right) =1,$ hence by Theorem XII of $\cite{Carmichael}$ (or by
theorem 2.2.4 (iii) of $\cite{Luca6}$), $11\mid L_{10}.~$Therefore $%
m_{11}\mid 10,~$i.e.~$n\mid 10$.~Since $n\geq 5$ is a prime, we get that $%
n=5.$

We conclude that $11$ is primitive divisor for $d=7.$

In particular, $u$ and $v$ are integers. Since $11$ is coprime to $%
-4dv^{2}=-28v^{2},$ we get that $v=\pm 7^{\alpha _{1}}.~$Since $%
y=u^{2}+7v^{2},$ we get that $u$ is even.

In the case $v=\pm 7^{\alpha _{1}}$, equation \textit{\eqref{eq:(3.5)}}
becomes 
\begin{equation*}
\pm 11^{\beta _{1}}=5u^{4}-70u^{2}v^{2}+49v^{4}.
\end{equation*}%
Since $u$ is even, it follows that the right hand side of the last equation
above is congruent to $1\pmod 8$. So $\pm 11^{\beta _{1}}\equiv 1\pmod 8$,
showing that the sign on the left hand side is positive and $\beta _{1}$ is
odd, or the sign on the left hand side is negative and $\beta _{1}$ is even.

Assume first that $\beta _{1}=2\beta _{0}+1$ be odd. We get 
\begin{equation*}
11V^{2}=5U^{4}-70U^{2}+49,
\end{equation*}%
where $(U,V)=({u}/{v},{11^{\beta _{0}}}/{v^{2}})$ is a $\{7\}$-integral
point on the above elliptic curve. We get that the only such points on the
above curve are $(U,V)=(\pm 7,\pm 28).~$This does not lead to solutions of
our original equation.

Assume now that $\beta _{1}=2\beta _{0}$ is even and we get that 
\begin{equation*}
V^{2}=5U^{4}-70U^{2}+49,
\end{equation*}%
where $(U,V)=({u}/{v},{11^{\beta _{0}}}/{v^{2}})$ is a $\{7\}$-integral
point on the above elliptic curve. With MAGMA, we get that the only such
point on the above curve are $(U,V)=(0,7)$. This does not lead to solutions
of our original equation.

We now recall that a particular instance of the Primitive Divisor Theorem
for Lucas sequences implies that, if $n\geq 5$ is prime, then $L_{n}$ always
has a prime factor except for finitely many \textit{exceptional triples} $%
(\varphi ,{\overline{\varphi },n)}$, and all of them appear in the Table 1
in \cite{Bilu} (see also \cite{Abouzaid}). These exceptional Lucas numbers
are called \textit{defective.}

Let us assume that we are dealing with a number $L_{n}$ without primitive
divisors. Then a quick look at Table 1 in \cite{Bilu} reveals that this is
impossible. Indeed, all exceptional triples have $n=5,7$ or $13$. The
defective Lucas numbers whose roots are in ${\mathbb{K}}={\mathbb{Q}}(i\sqrt{%
d})$ with $d=7$ and $n=5,7$ or $13$ appearing in the list \textit{%
\eqref{eq:(3.1)}} is $(\varphi ,~{\overline{\varphi }})=((1+i\sqrt{7}%
)/2,~(1-i\sqrt{7})/2)~$ for which $L_{7}=7,~L_{13}=-1.~$Furthermore, with
such a value for $\varphi ~$we get that $y=|\varphi |^{2}=2$. However, this
is not convenient since for us $\ x$ and $y$ are coprime so $y$ cannot be
even. For $n=5$ and $d=11,$we get $L_{5}=1$ and $y=3$ with $(\varphi ,~{%
\overline{\varphi }})=((1+i\sqrt{11})/2,~(1-i\sqrt{11})/2).~$Therefore the
equation is $x^{2}+C=3^{5}$, where $C=7^{\alpha }\cdot 11^{\beta }$, with $a$
even and $b$ odd. Since $11^{3}>3^{5},$ we have $b=1,$ and next that $a=0.~$%
But it doesn't yield an integer value for $x.$ The proof is completed.
\end{proof}

\begin{remark}
\label{Remark:2}We mention here why the method applied for the proof of
Lemma \ref{Lemma:3} does not apply when $\alpha ~$and$~x$ are odd, $\beta $
is even. In this case $d=7,$ the class number of $%
%TCIMACRO{\U{211a} }%
%BeginExpansion
\mathbb{Q}
%EndExpansion
(\sqrt{7}i)~$is $1$.~With $\omega =\frac{1+\sqrt{7}i}{2}$~a prime dividing $%
2 $, and $\omega ^{\prime }$~its conjugate, let us now write $(x+z\sqrt{7}%
i)=\omega ^{b}\omega ^{c}\xi $, where $\xi $ is an integer in $%
%TCIMACRO{\U{211a} }%
%BeginExpansion
\mathbb{Q}
%EndExpansion
(\sqrt{7}i)$ of odd norm, not divisible by $7$ and $\xi ^{\prime }$ its
conjugate.~As both $x$ and $z$ are odd and they are coprime, we may take $%
c=1 $, $b\geq $ $1$. Taking norms we get $y^{n}=$ $2^{b+1}\xi \xi ^{\prime }$%
, and it easily follows that $\xi =c^{n}$ and $b+1=k.n$.~Now we take $%
\varphi =2^{k-1}c$, $\wp =2\omega ^{n-2}$,~and then we have $x+z\sqrt{7}%
i=\wp \varphi ^{n}$. A way to look at the rest of argument why this case is
essentially different from the primitive divisors in Lucas sequences thing:
From $x+z\sqrt{7}i=\wp \varphi ^{n}$ and its conjugate it follows that%
\begin{equation*}
z=\frac{\wp \varphi ^{n}-{\overline{\wp }~\overline{\varphi }}^{n}}{2\sqrt{7}%
i}
\end{equation*}%
If $\wp $ is in $%
%TCIMACRO{\U{211a} }%
%BeginExpansion
\mathbb{Q}
%EndExpansion
$ then the right hand side is the n-th term of a Lucas sequence. As $z$ has
a very nice prime factorization $7^{p}11^{q}$ then theory of primitive
divisors will work. But in our case $\wp $ is not in $%
%TCIMACRO{\U{211a} }%
%BeginExpansion
\mathbb{Q}
%EndExpansion
.$ Hence the right side, while it is the n-th term of a recurrence sequence,
this is not a Lucas sequence, and does not have the nice divisibility
properties of Lucas sequences. That's why the method of \cite{Bilu}~fails in
our case.
\end{remark}

\begin{acknowledgement}
I thank Professor Benne de Weger for many helpful suggestions and generous assistance during the preparation of this paper and I also would like to thank him for his hospitality during my visit at the Faculty of Mathematics and Computer Science of Eindhoven University of Technology in 2009. I thank Professor Nikos Tzanakis and my PhD advisor Professor Ismail Naci Cangul for their valuable suggestions about the plan of the paper. I thank Professor Steve Donnelly for valuable suggestions about MAGMA computations.
\end{acknowledgement}

\bigskip

\bigskip

\begin{tabular}{l}
G\"{o}khan Soydan \\ 
Isiklar Air Force High School, \\ 
16039 Bursa, TURKEY \\ 
gsoydan@uludag.edu.tr%
\end{tabular}


\begin{thebibliography}{99}
\bibitem{Abouzaid} Abouzaid, M.,\textit{\ Les nombres de Lucas et Lehmer
sans diviseur primitif,} J. Th. Nombres Bordeaux\/ 18, no. 2, 299--313,
(2006).

\bibitem{Arif1} Abu Muriefah, F.~S., Arif, S. A., \textit{The Diophantine
equation }$x^{2}+3^{m}=y^{n}$, Int. J. Math. Math. Sci.\/ 21, no. 3,
619--620, (1998).

\bibitem{Arif2} Abu Muriefah, F.~S., Arif, S. A., \textit{The Diophantine
equation }$x^{2}+5^{2k+1}=y^{n}$, Indian J. Pure Appl. Math.\/ 30, no. 3,
229--231, (1999).

\bibitem{Abu1997} Abu Muriefah, F.~S., Arif, S. A., \textit{On the
Diophantine equation }$x^{2}+2^{k}=y^{n}$ , Int. J. Math. Math. Sci.\/ 20,
no. 2, 299--304, (1997).

\bibitem{FSAbu7} Abu Muriefah, F.~S., \textit{On the Diophantine equation }$%
x^{2}+5^{2k}=y^{n}$, Demonstratio Math.\textit{\/} 39, no. 2, 285--289,
(2006).

\bibitem{FSAbu3} Abu Muriefah, F.~S., Luca, F., Togb\'{e}, A., \textit{On
the Diophantine equation }$x^{2}+5^{a}13^{b}=y^{n}$, Glasgow Math. J.\/ 50,
no. 1, 175--181, (2008).

\bibitem{FSAbu8} Abu Muriefah, F.~S., Bugeaud, Y.,\textit{\ The Diophantine
equation }$x^{2}+C=y^{n}$\textit{, a brief overview}, Rev. Colombiana Math.
40, no. 1, 31--37, (2006).

\bibitem{FSAbu9} Abu Muriefah, F.~S., Luca, F., Siksek, S., Tengely, Sz, 
\textit{On the Diophantine equation }$x^{2}+C=2y^{n}$, Int. J. Number
Theory\ \/ 5, no., 1117--1128, (2009).

\bibitem{Berczes2} B\'{e}rczes, A., Brindza, B., Hajdu, L., \textit{On the
power values of polynomials}, Publ. Math. Debrecen\textit{\/} 53, no.3-4,
375--381, (1998).

\bibitem{Berczes} B\'{e}rczes, A., Pink, I., \textit{On the Diophantine
equation }$x^{2}+q^{2k}=y^{n}$, Archive der Math. (Basel) 91, no. 6,
505--517, (2008).

\bibitem{Bilu} Bilu,Y., Hanrot, G., Voutier, P. M., \textit{Existence of
primitive divisors of Lucas and Lehmer numbers. With an appendix by
M.Mignotte,} J.Reine Angew. Math. 539, 75--122, (2001).

\bibitem{Bosma} Bosma, W., Cannon, J., Playoust, C., \textit{The Magma
Algebra System I. The user language}, J. Symbolic Comput.\/ 24, no. 3-4,
235--265, (1997).

\bibitem{Bugeaud2} Bugeaud, Y., Mignotte, M., Siksek, S., \textit{Classical
and modular approaches to exponantial Diophantine equations II. The
Lebesque- Nagell equation}, Compositio. Math.\/ 142, no. 1, 31--62, (2006).

\bibitem{CDLPS} Cangul, I. N., Demirci, M., Luca, F., Pint\'{e}r, \'{A}.,
Soydan, G., \textit{On the Diophantine equation }$x^{2}+2^{a}11^{b}=y^{n}$,
Fibonacci Quart.\textit{\/} 48, no. 1, 39-46, (2010).

\bibitem{CDILS} Cangul, I. N., Demirci, M., Inam, I., Luca, F., Soydan, G., 
\textit{On the Diophantine equation }$x^{2}+2^{a}3^{b}11^{c}=y^{n}$, Math.
Slovaca, to appear.

\bibitem{CDST} Cangul, I. N., Demirci, M., Soydan, G., Tzanakis, N., \textit{%
On the Diophantine equation }$x^{2}+5^{a}11^{b}=y^{n}$, Functiones et
Approximatio Commentarii Mathematici\textit{\ }43, no.2, 209-225, (2010).

\bibitem{Carmichael} Carmichael, R. D., \textit{On the numerical factors of
the arithmetic forms }$\alpha ^{n}\pm \beta ^{n},$ The Annals of
Mathematics, 2 nd Ser. 15, no. 1/4\textbf{, }30-48, (1913-1914).

\bibitem{Cohn2} Cohn, J.~H.~E., \textit{The Diophantine equation }$%
x^{2}+2^{k}=y^{n}$, Arch. Math (Basel)\/ 59, no. 4, 341--344, (1992).

\bibitem{Cohn1} Cohn, J.~H.~E., \textit{The Diophantine equation }$%
x^{2}+C=y^{n}$, Acta Arith.\/ 65, no. 4, 367--381, (1993).

\bibitem{De Weger1} De Weger, B. M. M., \textit{The Weighted Sum of Two
Units Being a Square}, Indag. Math. 2 , no. 1, 243-262, (1990).

\bibitem{De Weger2} De Weger, B. M. M., \textit{Algorithms for diophantine
equations}, 65 of \ CWI Tract, Stichting Mathematisch Centrum voor Wiskunde
en Informatica, Amsterdam, 1989.

\bibitem{Gebel} Gebel, J., Hermann, E., Peth\H{o}, A., Zimmer, H.G., \textit{%
Computing all }$S$\textit{-integral points on elliptic curves}, Math. Proc.
Camb. Phil. Soc. 127, no. 3, 383--402, (1999).

\bibitem{Luca5} Goins, E., Luca, F., Togb\'{e}, A., \textit{On the
Diophantine equation }$x^{2}+2^{\alpha }5^{\beta }13^{\gamma }=y^{n}$, ANTS
VIII Proceedings: A. van der Poorten and A. Stein (eds.), ANTS VIII, Lecture
Notes in Computer Science 5011, 430--442, (2008).

\bibitem{Gyry} Gy\H{o}ry, K., \ Pink, I., Pint\'{e}r, \'{A}.,\textit{\ Power
values of polynomials and binomial Thue-Mahler equations}, Publ. Math.
Debrecen\textit{\/} 65, no. 3-4, 341--362, (2004).

\bibitem{Landau} Landau, F., Ostrowski, A., \textit{On the Diophantine
equation }$ay^{2}+by+c=dx^{n}$, Proc. London Math. Soc.\/ 19, no. 2,
276--280, (1920).

\bibitem{Le1} Le, M. H., \textit{On Cohn's conjecture concerning the
Diophantine equation }$x^{2}+2^{m}=y^{n}$, Arch. Math (Basel)\/ 78, no. 1,
26--35, (2002).

\bibitem{Lebesque} Lebesgue, V.~A., \textit{Sur I'impossibilit\'{e} en
nombres entieres de I'l equation }$x^{m}=y^{2}+1$, Nouvelles Ann. des. Math.%
\textit{\/} 9, no.1, 178--181, (1850).

\bibitem{Liqun1} Liqun, T., \textit{On the Diophantine equation }$%
x^{2}+3^{m}=y^{n}$, Integers: Electronic J. Combinatorial Number Theory\/ 8,
no. 1, 1-7, (2008)

\bibitem{Liqun2} Liqun, T., \textit{On the Diophantine equation }$%
x^{2}+5^{m}=y^{n}$, Ramanujan J.\/ 19, no. 3, 325-338, (2009)

\bibitem{Luca3} Luca, F., \textit{On a Diophantine equation}, Bull. Aus.
Math. Soc.\/ 61, no. 2, 241--246, (2000).

\bibitem{Luca1} Luca, F., \textit{On the equation }$x^{2}+2^{a}3^{b}=y^{n}$,
Int. J. Math. and Math. Sci.\/ 29, no. 4, 239--244, (2007).

\bibitem{Luca4} Luca, F., Togb\'{e}, A.,\textit{\ On the Diophantine
equation }$x^{2}+7^{2k}=y^{n}$, Fibonacci Quart.\/ 45, no.4, 322--326,
(2007).

\bibitem{Luca2} Luca, F., Togb\'{e}, A., \textit{On the Diophantine equation 
}$x^{2}+2^{a}5^{b}=y^{n}$, Int. J. Number Theory\/ 4, no. 6, 973--979,
(2008).

\bibitem{Luca6} Luca, F., \textit{Effective Methods for Diophantine Equations%
}, Winter School on Explicit Methods in Number Theory, Debrecen, Hungary,
January 26-30, 2009.

\bibitem{Luca7} Luca, F., Tengely, S., Togb\'{e}, A., \textit{On the
Diophantine equation }$x^{2}+C=4y^{n}$, Annales des sciences math\'{e}%
matiques du Qu\'{e}bec 33, no.2, 171-184, (2009).

\bibitem{Mignotte} Mignotte, M., De Weger, B.M.M., \textit{On the
Diophantine equations }$x^{2}+74=y^{5}$\textit{\ and }$x^{2}+86=y^{5}$,
Glasgow Math. J.\/ 38, no. 1, 77--85, (1996).

\bibitem{Nagell} Nagell, T., \textit{Contributions to the theory of a
category of diophantine equations of the second degree with two unknowns},
Nova Acta Reg. Soc. Upsal. \textit{Ser.\/} 4, no., 1--38, (1955).

\bibitem{Pink} Pink, I., \textit{On the Diophantine equation }$%
x^{2}+2^{a}3^{b}5^{c}7^{d}=y^{n}$, Publ. Math. Debrecen \/ 70 , no. 1-2,
149--166, (2007).

\bibitem{Tengely1} Tengley, Sz., On the Diophantine equation $%
x^{2}+a^{2}=2y^{p}$, Indagationes Math.-New Ser. 15, 291-304 (2004).
\end{thebibliography}
\end{document}